\documentclass[11pt]{imsart}
\usepackage{setup_normal}
\usepackage{newcommands}
\usepackage{float}
\usepackage{constants}
\usepackage{titletoc}
\usepackage{stackrel}
\usepackage{enumitem}
\usepackage{vwcol}
\usepackage{todonotes}
\newconstantfamily{c}{symbol=c}

\newenvironment{propbis}[1]
{\renewcommand{\thetheorem}{\ref*{#1}$'$}%
	\begin{proposition}}
	{\end{proposition}}

\begin{document}
	\begin{frontmatter}
		\title{Roughness of geodesics in Liouville quantum gravity}
		\runtitle{Roughness of geodesics in LQG}  
		\runtitle{Roughness of geodesics in LQG\;\;\;\;\;\;\;}

		\begin{aug}
			\author[A]{\fnms{Zherui} \snm{Fan}\ead[label=e1]{1900010670@pku.edu.cn}}
			\and
			\author[B]{\fnms{Subhajit} \snm{Goswami}\ead[label=e2]{goswami@math.tifr.res.in}}
			
			\runauthor{Fan Z. and Goswami S.}
			
			\affiliation{Peking University and Tata Institute of Fundamental Research 
			}

			\address[A]{School of Mathematical Sciences\\ 
				Peking University\\
				No.5, Yiheyuan Road\\
				Beijing, P.R.China 100087\\
				\href{mailto:1900010670@pku.edu.cn}{1900010670@pku.edu.cn}\\
				\phantom{as}}
			\address[B]{School of Mathematics\\
				Tata Institute of Fundamental Research\\
				1, Homi Bhabha Road\\
				Colaba, Mumbai 400005, India\\
				\href{mailto:goswami@math.tifr.res.in}{goswami@math.tifr.res.in}\\
				\phantom{ac}}
			
		\end{aug}

		\begin{abstract}
			The metric associated with the Liouville quantum gravity (LQG) surface has been constructed through a series of recent works 
			and several properties of its associated geodesics have been studied. In the current article we 
			confirm the folklore conjecture that the Euclidean Hausdorff dimension of LQG geodesics is stirctly greater than 1 for all values of the so-called Liouville first passage percolation (LFPP) parameter $\xi$.
			We deduce this from a general criterion due to Aizenman and Burchard \cite{AizBur99} which in our case amounts to {\em 
				near}-geometric bounds on the probabilities of certain crossing events for LQG geodesics in the number of 
			crossings. We obtain  such bounds using the axiomatic characterization of the LQG metric 
			after proving a special regularity property for the Gaussian free field (GFF). We also prove an 
			analogous result for the LFPP geodesics.
		\end{abstract}

\end{frontmatter}
	
	\section{Introduction}
	\subsection{Background and motivation}\label{sec:background}
	Liouville quantum gravity (LQG) is a one-parameter family of random fractal surfaces that was first studied by physicists in the 1980 as a class of canonical models of random two-dimensional Riemannian manifolds 
	\cite{Pol81, Dav88, distler1989conformal}. 
	We refer the reader to \cite{GwynneAMS20, berestycki2021gaussian} for a mathematical introduction to this topic. 
	
	\vspace{0.1cm}

	{\em Formally} speaking, a $\gamma$-Liouville quantum gravity ($\gamma$-LQG) surface is a random 
	``Riemannian manifold'' with Riemannian metric tensor $\e^{\gamma h} ds^2$ where $h$ is a variant of the 
	{\em Gaussian free field} (GFF) on some domain $U \subset \mathbb C$ and $\gamma \in (0, 2]$ is the 
	underlying parameter. This is of course not well-defined as the GFF is a random Schwartz distribution (see, e.g.,~\cite{She07, berestycki2021gaussian, werner2020lecture} for a comprehensive introduction to the 
	GFF). Rigorous mathematical investigation into this surface as a random metric measure 
	space was set off with the construction of the associated volume measure in \cite{DupShe11} which is a special instance of the 
	Gaussian multiplicative chaos (GMC) developed by Kahane in 1985 \cite{Kahane85} 
	(see~\cite{RhodesVargas14} for a current introduction to this subject). 
	
	\vspace{0.1cm}
	
	On the metric side, Miller and Sheffield constructed the $\sqrt{8/3}$-LQG metric through a series of 
	papers \cite{MSbrownianmapI,MSbrownianmapII,MSbrownianmapIII} where they established a 
	deep link with the {\em Brownian map}. 
	For a general $\gamma \in (0, 2)$ (the so-called {\em subcritical phase}), the metric was recently 
	constructed  as a culmination of several works due to Ding, Dub\'{e}dat, Dunlap, Falconet, Gwynne, Miller, Pfeffer and Sun \cite{DingDubDunFalc20, DubFalcGwynPfeSun20, GwynneMillerconcon2021, 
		GwynneMiller2021}. Their construction starts by producing candidate distance functions which are obtained 
	as subsequential limits \cite{DingDubDunFalc20} 
	of a family of random metrics 
	known as the {\em Liouville first passage percolation} (LFPP). The limiting metric is then shown to be unique 
	in two stages. Firstly, every possible subsequential limit is shown to be a measurable function of the GFF $h$ satisfying a list of axioms motivated by the natural properties and scaling behavior of the LFPP metrics \cite{DubFalcGwynPfeSun20}. Subsequently 
	these axioms are shown to uniquely characterize a random metric \cite{GwynneMillerconcon2021, GwynneMiller2021} on the plane. See the ICM article by Ding, Dub\'{e}dat and Gwynne 
	\cite{ding2021introduction} for a detailed overview of the construction of LQG using LFPP.
	
	
	\vspace{0.1cm}
	
	Associated with the LFPP metric is a parameter $\xi = \xi(\gamma)$ which gives a reparametrization of the 
	LQG metric although its explicit dependence on $\gamma$ is currently unknown (see~\cite{DingGwynne20, ding2021introduction}). There is yet another parametrization of LQG which is perhaps more popular in the 
	physics community, namely the {\em matter central charge} ${\mb c}_{{\rm M}}$. The subcritical phase 
	(i.e., $\gamma \in (0, 2)$) corresponds to ${\mb c}_{{\rm M}} \in (-\infty, 1)$ whereas the critical ($\gamma = 2$) and the supercritical phases ($\gamma$ complex, $|\gamma| = 2$) correspond to ${\mb c}_{{\rm M}}  = 
	1$ and ${\mb c}_{{\rm M}} \in (1, 25)$ respectively. See \cite{ding2021introduction} and \cite[Section~1]{ding2021uniqueness} to get a better understanding of the interplay between these 
	different parameters. Of these three phases, the supercritical phase is the most mysterious, not least because $\gamma$ is complex. However, one can still assign a LFPP parameter $\xi > 0$ to 
	such $\gamma$ \cite{ding2020tightness} and consider subsequential limits of LFPP metrics like before. This 
	program was carried out in a series of works by Ding, Gwynne and Pfeffer \cite{ding2020tightness, 
		pfeffer2021weak,ding2021uniqueness} and hence the LQG metric is now defined for all values of $\xi \in (0, 
	\infty)$. 
	
	\medskip
	
	In parallel to these works, many properties of the LQG metric and its associated geodesics (which are known to exist, see Section~\ref{sec:lqgdef} below) have been 
	studied and discovered in recent times; see \cite{ding2021introduction} for an overview. Of particular 
	relevance to the current article is the paper \cite{GwynPfefShef22} where it was proved (see~Theorem~1.8), 
	amongst several other results, that there is a deterministic constant $\Delta_{{\rm geo}} = \Delta_{{\rm geo}} 
	(\gamma) > 0$ for any $\gamma \in (0, 2)$ (equivalently $\xi \in (0, \xi_{{\rm crit}} \equiv \xi(2))$) such that a.s. the Euclidean Hausdorff dimension of every $\gamma$-LQG geodesic started from $0$ is equal to $\Delta_{{\rm geo}}$.  Let us note at this point that the dimension of a geodesic is always equal to 1 w.r.t. the 
	metric for which it is a geodesic, i.e., the LQG metric in this case. Not much is known about the precise value 
	of $\Delta_{{\rm geo}}$ except for an upper bound given in \cite[Corollary~1.10]{GwynnePfeffer19}. 
	
	\vspace{0.15cm}
	
	As to the {\em lower} bound on the dimension, it is expected ``$\ldots$that it is strictly greater than 1'' where the quoted text is from Ding, Dub\'{e}dat and Gwynne's ICM paper \cite{ding2021introduction} (see 
	Section~3.3).  A closely related question was addressed in a work by Ding and Zhang \cite{DZ19} where 
	they showed under a positivity assumption on the distance exponent for a variant of LFPP that the {\em Euclidean length} of any LFPP-geodesic joining two macroscopically separated points diverges, with high 
	probability, as a (negative) power of $\varepsilon$ as $\varepsilon \to 0$ where $\varepsilon$ is the 
	``regularization'' parameter for LFPP (see, e.g., \cite{ding2021introduction} and Section~\ref{subsec:LFPP_intro} below for a precise definition). But this property of the exponent is 
	currently known to hold only for very small values of $\xi$ from \cite{DingGos19} and for $\xi \in 
	(0.266\ldots, 1/\sqrt{2})$  from \cite[Theorem~2.3]{GwynnePfeffer19}. In the present article we 
	confirm this conjecture for the LQG geodesics corresponding to any value of $\xi \in (0, \infty)$. We also 
	extend the result of Ding and Zhang to {\em all} values of $\xi$ for a variant of LFPP.
	
	\subsection{Lower bound on the dimension of LQG geodesics}\label{subsec:LQG_geod_intro}
	Based on the discussion in the previous subsection, we know that the LQG metric is a measurable function 
	of the underlying field $h$ which we will refer to as $D_h$ in the sequel while keeping its dependence on 
	the LFPP parameter $\xi$ implicit. Whenever we speak of ``a $D$-geodesic'' for some metric $D$ in this paper, we mean a geodesic associated with $D$ connecting two points $z, w \in \mathbb C$ (see 
	Section~\ref{sec:lqgdef} below for a definition). We can now state our main result on LQG geodesics. 
	\begin{theorem}\label{thm:main}
		Let $h$ be 
		the whole-plane GFF and $D_h$ be the associated $\xi$-LQG metric. 
		Then for each $\xi > 0$, with probability 1,  the Euclidean Hausdorff dimension of any $D_h$-geodesic 
		is strictly greater than 1.
	\end{theorem}
	\begin{remark}[Other variants of GFF]\label{remark:GFF_variant}
		Although we stated Theorem~\ref{thm:main} for the whole-plane GFF only, the result also holds for zero 
		boundary GFF on proper subdomains of $\mathbb C$ by local absolute continuity. See 
		Section~\ref{sec:lqgdef} below.
	\end{remark}
	
	\smallskip
	
	Broadly speaking, the identifying characteristic of a rough random curve is its propensity to deviate from a 
	straight line segment in multiple scales. This can be manifested, for example, in the rapid decay of 
	simultaneous crossing probabilities of thin rectangles (or annuli) by the curve. Aizenman and Burchard 
	showed in their celebrated work \cite{AizBur99} that a geometric decay of such crossing probabilities 
	in the number of rectangles (see~\cite[display~(1.15)]{AizBur99}) is sufficient for the (Euclidean) Hausdorff dimension 
	to be greater than 1. Unfortunately, we can not prove this condition for LQG geodesics in full generality  and 
	it is not clear to us whether this is even true. However, it turns out that the proof of the lower bound on 
	Hausdorff dimension involves some specific configurations of rectangles (or annuli) for which we can 
	establish a near-geometric decay. To this end let us introduce:
	\begin{defn}\label{def:goodcollection}
		A collection of disks $\{A_j\}_{j=1}^{n}$ is said to be $(\lambda, \nu)$-{\em balanced} for some $\lambda, \nu 
		> 1$ if the following three conditions are satisfied. 
		\begin{enumerate}[itemsep= 1.5ex, label = {\bf C\arabic*}]
			\item For any $j \ne j' $, $d(A_j, A_{j'}) \ge \max({\rm diam}(A_j), {\rm diam} ({\rm A}_{j'}))$ where $d(\cdot, 
			\cdot)$ denotes the Euclidean distance.
			\item There exists a sequence of non-negative integers $k_1 < k_2 < \cdots < k_m$ and $L_0 > 0$ such that each ${\rm diam}(A_j) \in [L_0 \lambda^{-k_i}/100, 100L_0 \lambda^{-k_i}] \eqqcolon I_i$ for some $i 
			\in [1, m]$.
			\item $\left | \{j: {\rm diam}(A_j) \in I_i \} \right | \in [\nu/100, 100\,\nu]$ for all $i \in [1, m]$.
		\end{enumerate}
	\end{defn}

\smallskip

As we will see below, such families of disks arise naturally when we subdivide a path into smaller crossings in a certain hierarchical fashion. We just need a few more definitions before we can resume our discussion on Theorem~\ref{thm:main}. If $A$ is the 
(open) Euclidean disk $B(z, r)$ centered at $z$ with radius $r$, we abbreviate the (Euclidean) annulus $A(z, r/2, r) \coloneqq B(z, r) \setminus B(z, r/2)$ as $A_{\circ}$. 
We say that a path $P$ {\em crosses} $A_{\circ}$ if it has a segment that is contained in $\overline A$, has 
its both endpoints on $\partial A$ and intersects the circle $\partial B(z, r/2)$.

We can now state the main ingredient of our proof of Theorem~\ref{thm:main}.
\begin{proposition}\label{prop:crossingprob}
Let $h$ be as in the statement of Theorem~\ref{thm:main}. Then for any $\lambda_0 > 1$ and $\xi > 0$, 
there exist $\rho = \rho(\lambda_0, \xi) \in (0, 1)$ and a positive absolute constant $C$ such that the 
following holds. For any collection of disks $\{A_j\}_{j = 1}^n$ that is $(\lambda, \nu)$-balanced for some 
$\lambda \ge \lambda_0$ and $\nu > 1$, one has
\begin{equation}\label{eq:crossingprob}
	\P[ \text{{\rm a $D_h$-geodesic crosses $A_{j, \circ}$ for all $j$}} \,] \le C \rho^{\frac{m \sqrt{\nu}}{\log \e \nu}}.
\end{equation}
\end{proposition}
Theorem~\ref{thm:main} follows from Proposition~\ref{prop:crossingprob} by a similar argument 
as in \cite{AizBur99} with some adjustments along the way. For sake of completeness, we present 
the essential parts of the argument below. Let us start by recalling the notion of {\em straight runs} from 
\cite{AizBur99}. For a $\lambda > 1$ which we fix at the outset, a path $P$ in $\R^2$ is said to exhibit a straight run at scale $L$ if it traverses some rectangle of length $L$ and cross-sectional diameter 
$(9/\sqrt{\lambda})L$ in the ``length'' direction, joining the centers of the corresponding sides. Two straight 
runs are {\em nested} if one of the defining rectangles contains the other.

The straight runs of $P$ are $(\lambda, k_0)$-{\em sparse} if $P$ does not exhibit any nested collection of 
straight runs on a sequence of scales $L_{k_1} > \cdots > L_{k_m}$ with $L_k \coloneqq L_0 \lambda^{-k}$ and  $m \ge \frac12 \max\{k_m, k_0\} = \frac12 k_m$. We have the following deterministic result from \cite{AizBur99}.
\begin{theorem}\cite[Theorem~5.1]{AizBur99}\label{thm:hausdorff} 
If the straight runs of a given path $P$ are $(\lambda, k_0)$-sparse for some $k_0 > 0$, then the Hausdorff dimension of $P$ is at least $s$ with $s$ given by $\lambda^s = \sqrt{p(p+1)}$ and $p$ an 
integer strictly smaller than $\lambda$.	
\end{theorem}
We can now finish the
\begin{proof}[Proof of Theorem~\ref{thm:main}]
It suffices to show that there exist $p < \infty$ such that for every $\lambda > p$ and $K \subset \subset \R^2$, 
\begin{equation}\label{eq:haudorff_bnd}
\P [\,\text{straight runs of $D_h$-geodesics are $(\lambda, k_0)$-sparse in $K$}\,] \ge 1 - C_K 2^{-k_0}
\end{equation}
where $C_K$ depends only $K$ and $L_0 = 1$. Indeed, using Borel-Cantelli lemma we can conclude from 
\eqref{eq:haudorff_bnd} that there exists, almost surely, $k_0 = k_0(\omega) < \infty$ such that all the 
$D_h$-geodesics in $K$ are $(\lambda, k_0)$-sparse. Taking $\lambda$ close enough to $p$, we then get 
from Theorem~\ref{thm:hausdorff} that the Hausdorff dimension of any $D_h$-geodesic in $K$ is strictly greater than 1. Theorem~\ref{thm:main} now follows by letting $K \uparrow \R^2$.

\smallskip

Let us now return to the proof of \eqref{eq:haudorff_bnd}. To this end we first compute the probability that 
there is a nested sequence of straight 	runs at scales $L_{k_1}, \ldots, L_{k_m}$. If a path crosses a 
rectangle $R$ of length $L$ and width $(9/\sqrt{\lambda})L$ in the long direction, then it also crosses a 
rectangle $R'$ of width $(10  / \sqrt{\lambda})L$ and length $L/2$ centered at a line segment joining 
discretized points in $L'\Z^d$, provided that $L' \le L/\lambda$ is picked in a suitable way. Furthermore if $L 
= L_k$ and $R$ contains a smaller rectangle with length $L_{k'}$ and the same aspect ratio for some $k' > 
k$, then $R'$ can be chosen so that it contains the smaller rectangle as well. Therefore the number of 
possible locations of the $m$ nested rectangles obtained in this manner from the straight runs at scales 
$L_{k_1}, \ldots, L_{k_m}$ is bounded by
\begin{equation}\label{eq:noofloc}
	C_K \lambda^{4k_1}\lambda^{4(k_2 - k_1)} \cdots \lambda^{4(k_m - k_{m-1})} \le C_K \lambda^{4k_m}.
\end{equation}

\smallskip

Let us now fix a sequence $R_i, i = 1, \ldots, m$ of nested rectangles of length $L_{k_i}/2$ and width $(10 / 
\sqrt{\lambda})L_{k_i}$. In the sequel we will call a collection $\mathcal S$ of subsets of the plane as {\em 
	well-separated} if the distance between any $S \in \mathcal S$ and the rest is at least as large as the 
diameter of $S$ (recall the first condition in Definition~\ref{def:goodcollection}). Now split each of the rectangles $R_i$ to 
get  $\sqrt{\lambda}/40$ shorter rectangles of aspect ratio $2$. Since $R_{i+1}$ intersects at most two of 
the shorter rectangles obtained by subdividing $R_i$, the number of rectangles in a maximal well-separated collection is at 
least $m(\sqrt{\lambda}/ 80 - 2)$. Let us call these new rectangles $\{R_j'\}_{j = 1}^n$. Now observe that 
any path that crosses $R_j'$ also crosses $A_{j, \circ}$ (cf.~\eqref{eq:crossingprob}) where $A_j$ is the disk centered at the line joining the midpoints of the shorter sides of $R_j'$. Hence the 
probability of a geodesic crossing all the $R_i$'s is bounded above by the probability of crossing $A_{j, 
	\circ}$ for all $j \in [1, n]$. It also follows from the definition of $R_j'$ and $A_j$ that the family of disks 
$\{A_j\}_{j = 1}^n$ is $(\sqrt{\lambda},\sqrt{\lambda} )$-balanced for any $\lambda$ satisfying 
$\sqrt{\lambda} > 800$. Therefore using Proposition~\ref{prop:crossingprob} we get for any such 
$\lambda$ 
\begin{equation*}
	\P[\text{$R_1, \ldots, R_m$ are crossed by a $D_h$-geodesic}] \le C \rho^{\frac{m \lambda^{1/4}}{\log \lambda}}
\end{equation*}
where $\rho = \rho(\xi) \in (0, 1)$. Combined with \eqref{eq:noofloc} this yields that for all $\lambda$ larger 
than some fixed number and $k_m$ satisfying $m \ge k_m/2$,
\begin{equation*}
	\P\left[\begin{array}{c} \text{there exists a nested sequence of staright runs of }\\ \text{a $D_h$-geodesic at scales $L_{k_1}, \ldots, L_{k_m}$ inside $K$}\end{array}\right] \leq C_K \e^ {(C\log \lambda + \log \rho 
		\lambda^{1/4}/\log \lambda)m }
\end{equation*}
where $C$ is an absolute constant. Now choosing $p > (800)^2$ large enough so that $$C\log \lambda + \log \rho \lambda^{1/4}/\log \lambda < -4,$$
for all $\lambda > p$ and summing over all the sequences $k_1 < \cdots < k_m$ satisfying $m \le k_m \le 2m$, we get
\begin{equation*}
	\P\left[\begin{array}{c} \text{there exists a nested sequence of staright runs of a $D_h$-geodesic}\\ \text{ at (log-)scales ${k_1}, \ldots, k_m$ inside $K$ with $m \le k_m \le 2m$}\end{array}\right] \leq C_K 4^{m}\e^{-4m} \le C_K 4^{-m}.
\end{equation*}
Finally, summing over $m \ge \frac {k_0}{2}$ yields \eqref{eq:haudorff_bnd}.
\end{proof}

\subsection{Lower bound on the length of LFPP geodesics} \label{subsec:LFPP_intro}
Let us first give a definition of the {\em Liouville first passage percolation} (LFPP) that is standard in recent 
literature. Soon we will give another definition which is equivalent to the former one in the limiting sense (i.e., 
they both converge to the LQG metric) and is more convenient to work with for our purpose. In the remainder 
of this section, $h$ is a whole-plane GFF with the additive constant chosen so that its average over the unit 
circle is 0. The starting point in all definitions of LFPP is a family of continuous functions which approximate $h$. For $\varepsilon > 0$, consider a {\em mollified} version of $h$ by
\begin{equation*}
h_{\varepsilon}^{\star}(z) \coloneqq  (h \star p_{\varepsilon^2/2})(z) = \int_{\mathbb C} h(w) 
\,p_{\varepsilon^2/2}(z - w) \, dw, \quad \forall z \in \mathbb C,
\end{equation*}
where $p_s(z) \coloneqq \frac{1}{2\pi s}\exp\big( - \tfrac{|z|^2}{2s}\big)$ is the heat kernel and the 
integration is in the sense of distributional pairing.

Now we define the {\em Liouville first passage percolation} (LFPP) with parameter $\xi$ as the family of 
random metrics $\{D_h^{\varepsilon}\}_{\varepsilon > 0}$ defined by
\begin{equation}\label{eq:D_h_ep}
D_{h}^{\varepsilon} (z, w) \coloneqq \inf_{P: z \to w} \int_{0}^1 \e^{\xi h^{\star}_{\varepsilon}(P(t))} |P'(t)| dt, 
\quad \forall z, w \in \mathbb C
\end{equation}
where the infimum is over all piecewise $C^1$-paths $P: [0, 1] \mapsto \mathbb C$ from $z$ to $w$. To get 
a non-trivial limit of the metrics $D_h^{\varepsilon}$ in some suitable topology, one needs to re-normalize 
them. The standard, although somewhat arbitrary, choice for the normalizing factor is
\begin{equation}\label{eq:median}
\mathfrak{a}_{\varepsilon} \coloneqq \text{median of }\inf\left\{ \int_0^1 \e^{\xi h_{\varepsilon}^{\star}(P(t))} |P'(t)| dt:    P \mbox{ is a left-right crossing of }[0, 1]^2\right\}.
\end{equation}
The tightness of the rescaled metrics $\{\mathfrak a_{\varepsilon}^{-1}D_h^{\varepsilon}\}_{\varepsilon > 0}$ 
was established in \cite{DingDubDunFalc20} and \cite{ding2020tightness} for subcritical and general $\xi$ respectively albeit in different topologies. See \cite{ding2021introduction} for more details.

Since the function $h^{\star}_{\varepsilon}$ is continuous, the geodesics associated with  
$D_h^{\varepsilon}$ are (locally) rectifiable and hence their 
Hausdorff dimension is 1. Therefore we can instead look at the (Euclidean) lengths of geodesics and ask 
whether they diverge as $\varepsilon \to 0$ with high probability. Ding and Zhang \cite{DZ19} proved a 
power law divergence for the length when the LFPP distance is defined using  the discrete Gaussian free 
field (DGFF) under the assumption that, with high probability, maximum LFPP distance (appropriately scaled) between any points in 
a compact set is at most $\varepsilon^{c}$ for some $c > 0$. As already mentioned in Section~\ref{sec:background}, this 
condition is currently known to hold only for very small values of $\xi$ \cite{DingGos19} and $\xi \in 
(0.266\ldots, 1/\sqrt{2})$ \cite{GwynnePfeffer19}. Using similar ideas as involved in the proof of 
Theorem~\ref{thm:main},  we can deduce the power law lower bound on the length of LFPP geodesics for 
{\em all} values of $\xi$. We state this result for a slightly different choice of mollification as described below.

\vspace{0.1cm}

In several situations like in ours, it turns out to be more convenient to work with mollifications that depend 
{\em locally} on $h$ unlike $h_{\varepsilon}^{\star}(\cdot)$ above. 
To this end, let us consider
\begin{equation}\label{def:widehat_h_ep}
\widehat h_{\varepsilon}^\star(z) \coloneqq \int_{\mathbb C} \psi (\varepsilon^{-1/2}(z - w))\, 
h(w)\,p_{\varepsilon^2/2}(z - w)\, dw, \quad \forall z \in \mathbb C,
\end{equation}
where $\psi: \mathbb C \mapsto [0, 1]$ is a deterministic, smooth, radially symmetric bump function 
supported in $\overline B(0, 1)$ that is identically equal to 1 on $B(1/2, 0)$. We can define a LFPP metric 
$\widehat D_h^{\varepsilon}$ similarly as $D_h^{\varepsilon}$ with $\widehat h_{\varepsilon}^{\star}$ 
playing the role of $h_{\varepsilon}^{\star}$ (recall \eqref{eq:D_h_ep}). It was proved in \cite[Lemma~2.1]{DubFalcGwynPfeSun20} that a.s. $\lim_{\varepsilon \to 0} \tfrac{\widehat 
D_h^{\varepsilon}(z, w; U)}{D_h^{\varepsilon}(z, w; U)} = 1$ uniformly over all $z, w \in U$ ($z 
\ne w$) for each bounded open set $U \subset \mathbb C$ where $d(\cdot, \cdot ; U)$ is the internal metric 
of $d$ on $U$ (see Definition~\ref{def:metricspace} below). Consequently, the families of metrics $\{ \mathfrak a_{\varepsilon}^{-1} D_{h}^{\varepsilon}\}_{\varepsilon > 0}$ and $\{ \mathfrak a_{\varepsilon}^{-1} 
\widehat D_{h}^{\varepsilon}\}_{\varepsilon > 0}$ have the same (weak) subsequential limits 
\cite[Lemma~2.15]{pfeffer2021weak}. We now state our result on the (Euclidean) lengths of $\widehat 
D_h^{\varepsilon}$-geodesics. Below and in rest of the article, two points $z, w \in \mathbb C$ are said to be 
$\kappa$-{\em separated} if their Euclidean distance is at least $\kappa$.
\begin{theorem}\label{thm:LFPP_length}
For each $\xi \in (0, \infty)$, there exists $\alpha = \alpha(\xi) > 0$ such that for every $\kappa \in (0, 1)$ and 
$K \subset \R^2$ compact,
\begin{equation*}
\lim_{\varepsilon \to 0} \P [ \text{{\rm the length of any $\widehat D_h^{\varepsilon}$ geodesic 
		connecting two any $\kappa$-separated points in $K$}} \ge \varepsilon^{-\alpha} ] = 1.
\end{equation*}
\end{theorem}

\smallskip

\begin{remark}[Relationship between Theorem~\ref{thm:main} and Theorem~\ref{thm:LFPP_length}]\label{remark:relation}
The results given by Theorem~\ref{thm:main} and Theorem~\ref{thm:LFPP_length} are independent {a priori} 
since neither implies the other {\em even if} one assumes that the LFPP geodesics converge in Hausdorff 
distance to the LQG geodesics. Although the ideas behind their proofs are similar, the arguments involved 
vary significantly in their details.
\end{remark}

\smallskip

\begin{remark}[Quantitative lower bounds on the dimension and the length exponent]\label{remark:quantitaive}
It is possible to extract quantitative lower bounds on the (Euclidean) Hausdorff dimension and the length 
exponent for LQG and LFPP geodesics respectively from our proofs. In fact, it can be shown that the 
lower bound on the Hausdorff dimension is asymptotic to $1 + \e^{-C \xi^{-2}}$ as $\xi \to 0$ along with an 
analogous lower bound for the length exponent $\alpha$. However, we do not expect these bounds to be 
optimal.
\end{remark}

\bigskip

\smallskip

We now briefly describe the organization of this article. In Section~\ref{sec:lqgdef} we review the axiomatic 
characterization as well as some basic properties of the LQG metric which we will need in our proof. 
Section~\ref{sec:crossprob} is devoted to proving Proposition~\ref{prop:crossingprob} 
which comprises several intermediate lemmas including a certain regularity estimate for the harmonic 
extensions of GFF inside a collection of disks that is $(\lambda, \nu)$-balanced (Lemma~\ref{lem:max_dev}). 
Finally, in Section~\ref{sec:LFPP} we give the proof of Theorem~\ref{thm:LFPP_length}.

\smallskip

Our convention regarding constants is the following. Throughout, $c, c', C, C', \ldots$ denote positive 
constants that may change from place to place. Numbered constants are defined the first time they appear and remain fixed thereafter. 
Unless mentioned otherwise, all the constants are assumed to be absolute. Their dependence on other 
parameters, if any, will always be made explicit.

\bigskip

\noindent\textbf{Acknowledgements.} S.G.’s research was supported by the SERB grant SRG/2021/000032 
and in part by a grant from the Infosys Foundation as a member of the Infosys-Chandrasekharan virtual 
center for Random Geometry. We are grateful to Jian Ding for introducing the problem to us as well as many 
helpful discussions. We especially thank Ewain Gwynne for giving many valuable comments on an earlier draft of the paper.

%
%
%

\section{Definition and some properties of the LQG metric}\label{sec:lqgdef}
In this section, we review the definition and some basic properties of the $\xi$-LQG metric. Let us start 
by recalling some basic notions from metric geometry.
\begin{defn}\label{def:metricspace}
Let $(X, d)$ be a metric space, with $d$ allowed to take on infinite values. 

\begin{itemize}[itemsep = 1.5ex]
	\item For a path $P: [a, b] \mapsto X$, the $d$-$length$ of $P$ is defined by
	\begin{equation*}
		{\rm len}(P; d) \coloneqq 	\sup_T\sum_{i = 1}^{\# T} d(P(t_i), P(t_{i-1}))
	\end{equation*}
	where the supremum is over all partitions $T : a = t_0 < \cdots < t_{\#T} = b$ of $[a, b]$. Note that the $d$-length of a path may be infinite.
	
	\item We say that $(X, d)$ is a length space if for each $x, y \in X$ and each $\varepsilon > 0$, there exists a 
	path of $d$-length at most $d(x, y) + \varepsilon$ from $x$ to $y$. If $d(x, y) < \infty$ a path from $x$ to 
	$y$ of $d$-length exactly $d(x, y)$ is called a {\em geodesic}.
	
	\item For $Y \subset X$, the {\em internal metric of $d$ on $Y$} is defined by 
	\begin{equation*}
		d(x, y ; Y) \coloneqq \inf_{P \subset Y} {\rm len}(P; d), \forall x, y \in Y	
	\end{equation*}
	where the infimum is over all paths $P$ in $Y$ from $x$ to $y$. Note that $d(\cdot, \cdot; Y)$ is a metric on 
	$Y$, except that it is allowed to take infinite values.
	
	\item If $X \subset \mathbb C$ we say that $d$ is a {\em lower semicontinuous metric} if the function $(x, y) \mapsto d(x, y)$ is lower semicontinuous w.r.t. the Euclidean topology. We equip the set of lower 
	semicontinuous metrics on $X$ with the following topology (the so-called {\em Beer 
		topology} \cite{Beer1982}) on lower semicontinuous functions on $X \times  X$ and the associated Borel 
	$\sigma$-algebra. A sequence of functions $\{f_n\}_{n \in \N}$ converges in this topology to another function $f$ if and only if 
	\begin{enumerate}
		\item[(i)] Whenever $(z_n, w_n) \in X \times X$ with $(z_n, w_n) \to (z, w)$, we have $f(z,w) \le \displaystyle{\liminf_{n \to \infty} f_n(z_n, w_n)}$. 
		
		\item[(ii)] For each $(z, w) \in X \times X$, there exists a sequence $(z_n, w_n) \to (z, w)$ such that $f_n(z_n, w_n) \to f(z, w)$.	
	\end{enumerate} 
\end{itemize}
\end{defn}

\vspace{0.2cm}

We would like to emphasize at this point for the sake of clarity that all the paths we consider in the sequel 
are {\em Euclidean paths}, i.e., paths in the Euclidean topology.

\medskip

We now define the LQG metric with paramater $\xi > 0$. The following axiomatic characterization of the 
LQG metric is taken from \cite{ding2021uniqueness} which covers all the phases of LQG. Also see the earlier 
works \cite{millerQiangeodesics20, GwynneMiller2021} for closely related formulations. The definition 
involves an additional parameter $Q	 = Q(\xi) > 0$ whose exact functional dependence on $\xi$ is currently 
unknown. We refer the reader to \cite[Section~1.2]{ding2021uniqueness} and 
\cite[Section~2.3.1]{ding2021introduction} for a detailed discussion on the relationship of this parameter 
with the coupling constant $\gamma$ and the matter central charge ${\mb c}_{{\rm M}}$.
\begin{defn}\label{definition:LQGmetric} Let $\mathcal D'$ be the space of distributions (generalized functions) on $\mathbb C$ equipped with the usual weak topology. For $\xi > 0$, an LQG metric with 
parameter $\xi$ is a measurable function $h \mapsto D_h$ from  to the space of lower semicontinuous 
metrics on $\mathbb C$ with the following properties\footnote{The definition of $D$ on any zero measure subset of $\mathcal D'$ w.r.t. the law of any whole-plane GFF plus a continuous function is 
	inconsequential to us.}. Let $h$ be a {\em GFF plus a continuous function 
	on $\mathbb C$}: i.e., $h$ is a random distribution on $\mathbb C$ which can be coupled with a random 
continuous function $f$ in such a way that $h - f$ has the law of the whole-plane GFF. Then the associated metric $D_h = D_h^{(\xi)}$ satisfies the following axioms.

\smallskip

\begin{enumerate}[itemsep= 1.5ex, label = {\Roman*}]
	\item  \label{axiomlen}{\bf Length space.} Almost surely, $(\mathbb C, D_h)$ is a length space.
	
	\item \label{axiomloc}{\bf Locality.} Let $U \subset \mathbb C$ be a deterministic open set. The 
	$D_h$-internal metric $D_h(\cdot, \cdot ; U)$ is a.s. given by a measurable function of $h\vert_{U}$ (see, e.g., 
	\cite[Section~2.2]{ding2021uniqueness} for a precise meaning of this).
	
	\item \label{axiomweyl}{\bf Weyl scaling.} For a continuous function $f: \mathbb C \mapsto \R$, define 
	\begin{equation}\label{eq:weyl_scaling}
		(\e^{\xi f} \cdot D_h)(z, w) \coloneqq \inf_{P: z \to w} \int_{0}^{{\rm len}(P; D_h)} \e^{\xi f(P(t))} dt, \quad \forall z, w \in \mathbb C,
	\end{equation}
	where the infimum is over all $D_h$-rectifiable paths from $z$ to $w$ in $\mathbb C$ parametrized by 
	$D_h$-length (we use the convention that $\inf \emptyset = \infty$). Then a.s. $\e^{\xi f} \cdot D_h = D_{h + 
		f}$ for every continuous function $f: \mathbb C \to \R$.
	
	\item \label{axiomaffine}{\bf Affine coordinate change.} There is a specific choice of $Q = Q(\xi) > 0$ such that 
	for each fixed deterministic $r > 0$ and $z \in \mathbb C$, a.s. 
	\begin{equation}\label{eq:coordinate_change}
		D_h(ru +z, rv +z) = D_{h(r\cdot + z) + Q\log r}(u, v), \quad \forall u, v \in \mathbb C.
	\end{equation}
	
	\item \label{axiomfinite}{\bf Finiteness.} Let $U \subset \mathbb C$ be a deterministic, open, connected set and let $K_1, K_2 \subset U$ be disjoint, deterministic, compact, connected sets which are not singletons. 
	Almost surely, $D_h(K_1, K_2; U) < \infty$.
\end{enumerate}
\end{defn}

The following theorem \cite{ding2021uniqueness} (see also \cite{GwynneMiller2021}) asserts that the LQG 
metric as defined in Definition~\ref{definition:LQGmetric} exists and is unique.
\begin{theorem}\label{thm:lqgexistunique}
For each $\xi > 0$ there exists an LQG metric $D$ with parameter $\xi$ satisfying the axioms of 
Definition~\ref{definition:LQGmetric}. This metric is unique in the following sense. If $D$ and $\widetilde D$ 
are two LQG metrics with parameter $\xi$, then there is a deterministic constant $C > 0$ such that a.s. 
$\widetilde D_h = C D_h$ whenever $h$ is a whole-plane GFF plus a continuous function.
\end{theorem}
In view of Theorem~\ref{thm:lqgexistunique}, we can refer to the unique metric satisfying 
Definition~\ref{definition:LQGmetric} as the LQG metric (with parameter $\xi$). To be precise, the metric is 
unique only up to a global deterministic multiplicative constant. When referring to the LQG metric, we fix the 
constant in some arbitrary way. For example, we could require that the median distance between the left and 
right sides of $[0, 1]^2$ is $1$ when $h$ is a whole-plane GFF normalized so that its average over the unit 
circle is zero. It is clear that the geodesics underlying the metric, whenever they exist, remain invariant with 
respect to the choice of this constant.

\smallskip

Please note that the metric $D_h$ given by Theorem~\ref{thm:lqgexistunique} implicitly depends on our 
particular choice for the normalization of $h$. Indeed, it follows from the Weyl scaling 
(Axiom~\ref{axiomweyl}) that the metric corresponding to any other choice of normalization (say, we set the 
average of $h$ to be 0 on a different circle) is related to $D_h$ by a random positive prefactor.

\smallskip

The LQG metric associated with variants of GFF on other domains (like the zero-boundary GFF on a proper subdomain of $\mathbb C$) can be constructed from $D_h$ via restriction and / or local absolute continuity 
in view of locality (Axiom~\ref{axiomloc}) and the Weyl scaling; see \cite[Remark~1.5]{GwynneMiller2021}.

\smallskip

Of special interest are two types of $D_h$-distances which we now introduce:
\begin{defn}\label{def:annulus_dist}
For an annulus $A \subset \mathbb C$, we define $D_h(\text{across $A$})$ to be the $D_h$-distance 
between the inner and outer boundaries of $A$. We define $D_h(\text{around $A$})$ to be 
the infimum of the $D_h$-lengths of paths in $A$ which separate the inner and outer boundaries of $A$.
\end{defn}
These two types of (random) distances and the events {\em comparing} them play a crucial role in the 
study of LQG metric (see~\cite{ding2021introduction} and the references therein) as they are going to do in 
our work as well. Notice that these distances are determined by the internal metric of $D_h$ on $A$. It is a.s. the case that $D_h(\text{across $A$})$ and $D_h(\text{around $A$})$ are finite and positive 
(see~\cite{pfeffer2021weak} for more explicit tail bounds and also the discussion following 
\cite[Lemma~2.6]{ding2021uniqueness}). 

\smallskip

Finally, we come to the question of existence of geodesics. In the supercritical case $\xi > \xi_{{\rm crit}}$, 
the limiting metric in Theorem~\ref{thm:lqgexistunique} does not induce the Euclidean topology on $\mathbb C$. Rather, there exists an uncountable, Euclidean-dense set of {singular points} $z \in \mathbb 
C$ such that $D_h(z, w) = \infty$ for all $w \in \mathbb C \setminus \{z\}$ \cite[(1.8)]{ding2021uniqueness}. 
However, for each fixed $z \in \mathbb C$, a.s. $z$ is a non-singular point and hence the set of singular 
points has zero Lebesgue measure \cite{ding2021uniqueness}. On the subspace of non-singular points, we have the following result from \cite[Proposition~1.12]{pfeffer2021weak}.
\begin{prop}[\cite{pfeffer2021weak}]\label{prop:geodesic_existence}
Almost surely, the metric $D_h$ is complete and finite-valued on $\mathbb C \setminus \{\text{singular points}\}$. Moreover, every pair of points in $\mathbb C \setminus \{\text{singular points}\}$ can be joined by a $D_h$-geodesic 
\end{prop}
Furthermore, the geodesics connecting two given points (which exsist almost surely) is also a.s. unique. In fact, we have the following slightly more general statement from \cite[Lemma~2.7]{ding2021uniqueness}.
\begin{lemma}[\cite{ding2021uniqueness}]\label{lem:uniqgeodesic}
Let $K_1, K_2 \subset  \mathbb C$ be deterministic disjoint Euclidean-compact sets. Almost surely, there is a unique $D_h$-geodesic from $K_1$ to $K_2$.
\end{lemma}

\section{Simultaneous crossings of annuli}\label{sec:crossprob}
In this section we will prove Proposition~\ref{prop:crossingprob}. Let us start with the simplest scenario  
where there is only one disk, i.e., $n = 1$ in which case Proposition~\ref{prop:crossingprob} is implied by the 
following result.
\begin{lemma}\label{lem:crossingonedisk}
For any $z \in C$ and $r > 0$, let $F_{z, r} \coloneqq \{ \text{a $D_h$-geodesic crosses $A(z, r/2, 
	r)$}\}$. Then there exists $c'(\xi) \in (0, 1)$ such that
\begin{equation}\label{eq:crossingonedisk}
	\P[ \text{{\rm a $D_h$-geodesic crosses $A(z, r/2, r)$}}\,] \le c'(\xi).
\end{equation}
\end{lemma}

\begin{proof}
The argument presented below is similar to the one given in the proof of 
\cite[Lemma~4.3]{millerQiangeodesics20}  and will be used as a reference for our later arguments. We will 
bound $\P[F_{z, r}]$ by invoking a new event $G_{z, r}$ which involves a comparison between two special types of $D_h$-distances. To this end let us define 
\begin{equation*}
	G_{z, r} \coloneqq \{D_h(\text{around $A(z, 3r/4, 7r/8)$})  <  D_h(\text{across $A(z,r/2,  5r/8)$})\}.
\end{equation*}
and  observe that $G_{z, r} \subset F_{z, r}^c$. Indeed, on the event $G_{z, r}$, we can reroute any path 
$P$ crossing $A(z, r/2, r)$ through a path separating $A(z, 3r/4, 7r/8)$ such that the resulting path $P'$ has 
(strictly) smaller $D_h$-length. In particular {\em no} geodesic can cross $A(z, r/2, r)$ on the event $G_{z, r}$. 

\smallskip

As we now explain, probability of the event $G_{z, r}$ is independent of $z$ and $r$, i.e., $\P[G_{z, r}] = 
\P[G_{0, 1}]$. To see this notice that due to the Weyl scaling (Axiom~\ref{axiomweyl}) adding a (random) 
constant to $h$ only changes the metric $D_h$ by a multiplicative constant and consequently the 
event $G_{z, r}$ is a.s. determined by $h - h_r(z)$ where $h_r(z)$ is the average of $h$ over the circle 
$\partial B(z, r)$ (see~\cite[Section~3.1]{DupShe11} for an introduction to the circle average processes). Since the whole-plane GFF satisfies, for any $z \in \mathbb C$ and $r > 0$, 
\begin{equation}\label{eq:scale_trans_invar}
	h(r \cdot + z) - h_r(z) \stackbin{{\rm law}}{=} h - h_1(0)
\end{equation}
(see, e.g.~\cite[Section~2.2.3]{ding2021introduction}), we can immediately deduce $\P[G_{z, r}] = \P[G_{0, 
	1}]$. 

Hence it suffices to show that $\P[G_{0, 1}] > 0$. To this end let us recall from the previous section 
that both $D_h(\text{around $A(3/4, 7/8)$})$ and $D_h(\text{across $A(1/2, 5/8)$})$ are finite, positive 
random variables and hence there exists $C = C(\xi) \ge 1$ such that
\begin{equation}\label{eq:basic_comparison}
	\P[D_h(\text{around $A(3/4, 7/8)$})  <  C D_h(\text{across $A(1/2, 5/8)$})] > \frac 12
\end{equation}
where $A(r, R) \coloneqq A(0, r, R)$. Now consider a non-negative, radially symmetric (bump) function $\phi \in C_c^{\infty}(\mathbb C)$ 
supported in $B(0, 3/4)$ which is equal to 1 on $B(0, 5/8)$. As $(\mathbb C, D_h)$ is a length space 
(Axiom~\ref{axiomlen}), $D_h(\text{across $A$})$ is determined by the internal metric of $D_h$ on $A$ and 
so is $D_h(\text{around $A$})$ by definition. Hence, with $C$ as in the last display, we get for the metric 
$D_{h + \xi^{-1}\log C \phi}$,	
\begin{equation*}
	\P[D_{h + \xi^{-1}\log C \phi }(\text{around $A(3/4, 7/8)$})  <  D_{h + \xi^{-1} \log C \phi}(\text{across $A(1/2, 5/8)$})] > \frac 12.
\end{equation*}
Indeed, by the Weyl scaling, the $D_{h + \xi^{-1}\log C \phi}$-internal metric is $\e^{\xi \cdot \xi^{-1}\log C} = 
C$ times the $D_h$-internal metric inside $A(1/2, 5/8)$ whereas it is same as the 
$D_h$-internal metric in $A(3/4, 7/8)$ which yields the above bound in view of 
\eqref{eq:basic_comparison}. However, the laws of $h$ and $h + \phi$ are mutually absolutely continuous 
(see, e.g., \cite{MillerSheffield_Imaginary12016} for a proof) and consequently there exists $c'(\xi) < 1$ such that
$$\P[G_{0, 1}] = \P[G_{z, r}] \ge 1 - c'(\xi)$$ which gives us $\P[F_{z, r}] \le c'(\xi)$, i.e., 
\eqref{eq:crossingonedisk}.
\end{proof}

\smallskip

Now we come to main part of the proof where we have to deal with multiple disks. To this end let us 
consider the collection of disks $\{A_j\}_{j = 1}^n \equiv \{B(z_j, r_j)\}_{j = 1}^n$ that is $(\lambda, 
\nu)$-balanced for some $\lambda \ge \lambda_0 > 1$ and $\nu > 1$. In order to obtain a near-geometric decay as in \eqref{eq:crossingprob}, we will produce two collections of events $\{H_{z_j, r_j}\}_{j = 
1}^n$ and $\{\widetilde F_{z_j, r_j}\}_{j = 1}^n$ satisfying the following properties for all $j \in [1, n]$.
\smallskip
\begin{enumerate}[itemsep= 1.5ex, label = {\bf P\arabic*}]
\item \label{property1} $F_{z_j, r_j} \subset \widetilde F_{z_j, r_j} \cup H_{z_j, r_j}$.
\item \label{property2}  $\widetilde F_{z_j, r_j}$ is independent of the $\sigma$-algebra generated by 
$\big(\{H_{z_i, r_i}\}_{i = 1}^n, \{\widetilde F_{z_i, r_i}\}_{i \ne j}\big)$.
\item\label{property3} There exists $c'(\xi) \in (0, 1)$ such that $\P[\widetilde F_{z_j, r_j}] \le c'(\xi)$.
\end{enumerate}
\smallskip
Assuming that two families of events exist satisfying \ref{property1}--\ref{property3}, we can immediately 
write with $\mathcal S \coloneqq \{j \in [1, n]:  H_{z_j, r_j}^c \text{ occurs}\} \subset [1, n]$,
\begin{equation}\label{eq:exploration}
\begin{split}
	\P\Big[\, \bigcap_{j = 1}^n F_{z_j, r_j}   \Big]\, \stackrel{\ref{property1}}{\le} \, \P\Big[\bigcap_{j \in \mathcal S}\widetilde F_{z_j, r_j} \Big]\,  \stackrel{\ref{property2} + \ref{property3}}{\le} \, \E[ c'(\xi)^{|\mathcal S|} ] 
\end{split}
\end{equation}
where in the 
second step we also used the fact that the random set $\mathcal S$ is determined by the events $\{H_{z_j, 
r_j}\}_{j = 1}^n$. This leads to the bound
\begin{equation}\label{eq:bnd_F}
\P\Big[\, \bigcap_{j = 1}^n F_{z_j, r_j}  \Big]  \le c'(\xi)^{K} + \P[|\mathcal S| < K]
\end{equation}
for any $K \in [1, n]$ which yields Proposition~\ref{prop:crossingprob} provided we also have
\begin{equation}\label{eq:tail_bnd_S}
\P[ |\mathcal S| < c n] \le C \e^{- c(\lambda_0) \frac{m \sqrt{\nu}}{\log \e \nu}}
\end{equation}
(recall $m$ and also that $n \ge m\nu/100$ from Definition~\ref{def:goodcollection}).

\smallskip

Having laid out the basic strategy, we now proceed to defining the events $\{\widetilde F_{z_j, r_j}\}_{j = 1}^n$ and $\{H_{z_j, r_j}\}_{j = 1}^n$ to which end we will use the {\em Markov decomposition} of $h$ as 
stated below (see, e.g., \cite{She07, werner2020lecture, berestycki2021gaussian}  and \cite[Section~3.2]{SheffieldWelding2016} for a brief review on the whole-plane GFF in particular). 

\vspace{0.1cm}

\noindent For any $z \in \mathbb C$ and $r > 0$, $h$ can be decomposed as 
$$h = \widetilde h_{z, r} + \widehat h_{z, r}$$ where $\widetilde h_{z, r}$ is a GFF on $B(z, r)$ with zero 
boundary condition and $\widehat h_{z, r}$ is a distribution modulo absolute constant that is {\em harmonic} on 
$B(z, r)$ and agrees with $h$ in $\mathbb C \setminus B(z, r)$. Furthermore, $\widetilde h_{z, r}$ is 
independent of $\mathcal F_{z, r}$ --- the $\sigma$-algebra generated by $h|_{\mathbb C \setminus B(z, r) }$ (see, e.g., \cite[(2.2)]{ding2021uniqueness} for a precise definition). Note that $\widehat h_{z, r}$ is 
measurable w.r.t. $\mathcal F_{z, r}$. 

\smallskip

At this point we would like to draw the reader's attention to the fact 
that $h$ is treated in this decomposition as a distribution modulo additive constant, i.e., without any 
particular choice of normalization. In order to use the aforementioned independence between  $\mathcal F_{z, r}$ and $\widetilde h_{z, r}$ {\em after} normalizing $h$ so that $h_s(w) = 0$ for some $s > 0$ and $w \in \mathbb C$ (recall that $s = 1$ and $w = 0$ in our 
preexisting choice), we would need $B(z, r)$ to be disjoint from $\partial B(w, s)$. However, changing the normalization to $h_s(w) = 0$ only amounts to subtracting the (random) constant $h_s(w)$ from $h$ and 
hence, by the Weyl scaling, changes $D_h$ by the (random) factor $\e^{-\xi h_s(w)}$. This does not affect 
the relative distances nor the geodesics and hence the events $F_{z, r}$ and $G_{z, r}$ are invariant w.r.t. 
the choice of the circle on which we normalize $h$. Therefore, in the rest of the paper we assume that 
$\partial B(w, s)$ is {\em disjoint} from the union of $B(z_j, r_j); j \in [1, n]$ so that we can use the Markov 
property inside each $B(z_j, r_j)$ with $h$ normalized in this manner. 

\smallskip

Now for some $M > 0$ whose precise value would be chosen in Lemma~\ref{lem:max_dev} below, let us 
define the events $\widetilde F_{z, r} = \widetilde F_{z, r; M}$ and $H_{z, r} = H_{z, r; M}$ as follows:
\begin{equation}\label{def:tildeF_H}
\begin{split}
	\widetilde F_{z, r} &\coloneqq 	\big\{ D_{\widetilde h_{z, r}}(\text{around $A(z, 3r/4, 7r/8)$})  \ge
	\e^{-\xi M} D_{\widetilde h_{z, r}}(\text{across $A(z, r/2, 5r/8)$})\big\}, \text{ and}\\
	H_{z, r} & \coloneqq \big\{\sup_{u, v \in B(z, 7r/8)} \, |\widehat h_{z, r} (u) -  \widehat h_{z, r}(v) |  \ge M \big\}
\end{split}
\end{equation}
where $D_{\widetilde h_{z, r}}$ is the LQG metric on $B(z, r)$ associated with $\widetilde h_{z, r}$ 
(recall the discussion from Section~\ref{sec:lqgdef}). We then have:
\begin{lemma}\label{lem:tildeF_H}
Consider a collection of disjoint disks $\{B(z_j, r_j)\}_{j = 1}^n$. Then the families of events $\{\widetilde 
F_{z_j, r_j}\}_{j = 1}^n$ and $\{H_{z_j, r_j}\}_{j = 1}^n$ defined as in \eqref{def:tildeF_H} satisfy 
properties~\ref{property1}--\ref{property3} with $c'$ depending only on $\xi$ and $M$.
\end{lemma}
\begin{proof}
We verify each of the properties \ref{property1}--\ref{property3} below for the collections $\{\widetilde 
F_{z_j, r_j}\}_{j = 1}^n$ and $\{H_{z_j, r_j}\}_{j = 1}^n$.

\vspace{0.3cm}

\noindent{\em Property~\ref{property1}.}	
It immediately follows from \eqref{def:tildeF_H} and the Weyl scaling that $\widetilde F_{z, r}^c \cap H_{z, 
	r}^c \subset F_{z, r}^c$ (recall the argument previously given for $G_{z, r} \subset F_{z, r}^c$ in the proof of Lemma~\ref{lem:crossingonedisk}). 

\vspace{0.25cm}

\noindent{\em Property~\ref{property2}.} Since $\widetilde F_{z, r}$ is determined by $\widetilde h_{z, r}$ and the disks $B(z_j, r_j)$'s are disjoint, we get that the events $\{H_{z_i, r_i}\}_{i = 1}^n, \{\widetilde F_{z_i, 
	r_i}\}_{i \ne j}$ are measurable relative to $\mathcal F_{z_j, r_j}$. The Markov property then gives us the required independence between $\widetilde F_{z_j, r_j}$ and $\big(\{H_{z_i, r_i}\}_{i = 1}^n, \{\widetilde F_{z_i, r_i}\}_{i \ne j}\big)$ for all $j \in [1, n]$. 

\vspace{0.25cm}

\noindent{\em Property~\ref{property3}.} For the upper bound on the probability of $\widetilde F_{z, r}$, first 
notice that due to the scale and translation invariance of 
$h$ (recall \eqref{eq:scale_trans_invar}) as well as the Weyl scaling we have $\P[\widetilde 
F_{z, r}] = \P[\widetilde F_{0, 1}]$. It is known from \cite[Lemma~4.1]{millerQiangeodesics20} that the laws 
of $h|_{B(0, 7/8)} - h_1(0)$ and $\widetilde h_{0, 1}|_{B(0, 7/8)}$ are mutually absolutely continous and 
hence both $D_{\widetilde h_{0, 1}}(\text{around $A(3/4, 7/8)$})$ and $D_{\widetilde h_{0, 1}}(\text{across $A(1/2, 5/8)$})$ are finite and positive random variables (note that they are both determined by the internal 
metric of $D_{\widetilde h_{0, 1}}$ on  $B(0, 7/8)$). Given this fact, we can deduce an upper bound on 
$\P[\widetilde F_{0, 1}]$ (and hence $\P[\widetilde F_{z, r}]$) depending only on $\xi$ and $M$ by the 
exact same argument as that for the upper bound on $\P[G_{0, 1}]$. 
\end{proof}

Finally, it remains to check whether \eqref{eq:tail_bnd_S} also holds for these events.
\begin{lemma}\label{lem:max_dev}
Suppose $\{B(z_j, r_j)\}_{j = 1}^n$ is $(\lambda, \nu)$-balanced for some $\lambda \ge \lambda_0 > 1$ and 
$\nu > 1$ and the events $\{H_{z_j, r_j}\}_{j = 1}^n$ are defined by \eqref{def:tildeF_H}. Then there exists an 
absolute constant $M > 0$ such that with $\mathcal S$ defined as in 
\eqref{eq:exploration}, we have
\begin{equation}\label{eq:max_dev}
	\P [|\mathcal S| \le n/2] \le C \e^{- c(\lambda_0) \frac{m \sqrt{\nu} }{\log \e \nu}} .
\end{equation}
\end{lemma}
Combined with \eqref{eq:bnd_F} and Lemma~\ref{lem:tildeF_H}, this finishes the proof of 
Proposition~\ref{prop:crossingprob}. The proof of Lemma~\ref{lem:max_dev} will be given in the next 
subsection. \qed

\subsection{Proof of Lemma~\ref{lem:max_dev}}
Notice that Lemma~\ref{lem:max_dev} is essentially a statement about the regularity of harmonic extensions. In the case when the sets $A_j$'s concentric annuli, such estimates have been used many times 
in the LQG literature; see, e.g., \cite[Proposition~4.3]{millerQiangeodesics20}. However, our situation is 
quite different from theirs and necessitates a different approach. One key ingredient in the proof is the following variance estimate.
\begin{lemma}\label{lem:variance_estimate}
Let $z \in \mathbb C$, $r > 0$ and $s \in (0, 1)$. Then there exists $C = C(s) > 0$ such that for any two 
points $u, v \in B(z, sr)$,
\begin{equation}\label{eq:squared_diff}
	\E[ (\widehat h_{z, r}(u) - \widehat h_{z, r}(v))^2 ] \le C \Big( \frac{|u - v|}{r}\Big)^2.
\end{equation}
\end{lemma}
Estimates of this flavor are prevalent in the literature; see,~e.g., \cite[Lemma~3.10]{BDZ16}  for the 
analogous result in the context of a discrete GFF on a box and also \cite[Section~4.1]{werner2020lecture} 
which gives a weaker bound for the zero-boundary GFF on a proper subdomain of $\mathbb C$. Using it we 
can now finish the
\begin{proof}[Proof of Lemma~\ref{lem:max_dev}]
Let $\Delta_{z, r}(u, v)$ denote the difference $\widehat h_{z, r}(u) - \widehat h_{z, r}(v)$. In view of the 
variance bound \eqref{eq:squared_diff}, it follows from the Fernique's inequality \cite{Fer75} (see also 
\cite[Lemma~3.5]{BDZ16})  that
\begin{equation}\label{eq:max_exp_single_ball}
\E \, \big[  \sup_{u, v \in B(z, 7r/8)} \Delta_{z, r}(u, v) \big] \le C.
\end{equation}
This implies in particular,
\begin{equation}\label{eq:exp_max_sum_bnd}
	\begin{split}
		\mu \coloneqq \E \, &\big [ \sup_{(u_j, v_j) \, \in B(z_j, 7r_j/8)} \, \sum_{j} \,  \Delta_{z_j, r_j}(u_j, v_j) \big ] = \E \, \big [ \sum_{j} \, \sup_{(u_j, v_j) \, \in B(z_j, 7r_j/8)} \Delta_{z_j, r_j}(u_j, v_j) \big ]\\ &\le C n \le C' m \nu
	\end{split}
\end{equation}
where in the final step we used the bound $n \le 100 m \nu$ as a consequence of the definition of $(\lambda, 
\nu)$-balanced sets. Now suppose that we also have, for any choice of pairs of points $(u_j, v_j) 
\in B(z_j, 7r_j/8)$,
\begin{equation}\label{eq:variance_of_sum}
	v \coloneqq \var\,\big[\, \sum_{j} \Delta_{z_j, r_j}(u_j, v_j)   \big] \le C(\lambda_0)  m \nu^{3/2} \log \e \nu.
\end{equation}
Then, using the Borell-Tsirelson inequality \cite{Bor75, SCs74} (see, e.g., \cite[Theorem~2.1.1]{AT07}) we 
get
\begin{equation}\label{eq:borel_tsierl}
	\P\,\big[  \sum_{j} \, \sup_{(u_j, v_j) \, \in B(z_j, 7r_j/8)} \Delta_{z_j, r_j}(u_j, v_j) \ge 
	\mu + k\big] \le \e^{- k^2/2v} \le \e^{- c(\lambda_0)\, \frac{k^2}{m \nu^{3/2} \log \e \nu}}
\end{equation}
for all $k \ge 0$. Setting $M = 4C$ where $C$ is from \eqref{eq:exp_max_sum_bnd}, we can deduce from this:
\begin{equation*}
	\begin{split}
		\P[ |\mathcal S| \le n/2] &= \P \Big[  \# \big\{j:   \sup_{(u_j, v_j) \, \in B(z_j, 7r_j/8)} \, \sum_{j} \,  \Delta_{z_j, r_j}(u_j, v_j) \ge 4C\big\} > n/2 \Big] \\
		& \le \P\,\big[  \sum_{j} \, \sup_{(u_j, v_j) \, \in B(z_j, 7r_j/8)} \Delta_{z_j, r_j}(u_j, v_j) \ge 2 \mu + Cn  \big] \le \e^{c(\lambda_0)\frac{n^2}{m \nu^{3/2} \log \e \nu}} \le \e^{-c(\lambda_0) \frac{m \sqrt{\nu}}{ \log \e \nu}}
	\end{split}
\end{equation*}
where in the last step used the lower bound $n \ge m \nu / 100$ as implied by Definition~\ref{def:goodcollection}.

\medskip

Let us now verify the bound	in \eqref{eq:variance_of_sum} to which end we start by expanding
\begin{equation}\label{eq:var_expand}
	\begin{split}
		\var&\, \big[  \sum_{j} \Delta_{z_j, r_j}(u_j, v_j) \big] \\
		&\le \sum_{j} \var [ \Delta_{z_j, r_j}(u_j, v_j)] + 2\sum_{i \in [1, m]}\,\sum_{j \in S_i}\, \sum_{\ell \ge i}\,\sum_{\substack{k \in S_{\ell}, \\ k \ne j}} \big | \cov[ \Delta_{z_j, r_j}(u_j, v_j),  \Delta_{z_k, r_k}(u_{k}, v_{k})] \big |
	\end{split}
\end{equation}
where $S_i \coloneqq \big\{j:  2r_j \in [L_0 \lambda^{-k_i}/100,  100 L_0\lambda^{-k_i}] \big\}$ (recall 
Definition~\ref{def:goodcollection}). It follows from \eqref{eq:squared_diff} that
\begin{equation}\label{eq:variance_bnd}
	\var [ \Delta_{z_j, r_j}(u_j, v_j) ] \le C. 
\end{equation}
For the covariance terms first observe that due to the well-separatedness of $B(z_j, r_j)$'s, we have $2d_{j, k} \coloneqq |z_j - z_k| \ge 2 (r_j \vee r_k)$ whenever $j \ne k$ and hence the two {\em disjoint} disks 
$B(z_j, d_{j, k})$ and $B(z_k, d_{j, k})$ contain $B(z_j, r_j)$ and $B(z_k, r_k)$ respectively. Therefore we 
can write
\begin{equation*}
	\Delta_{z_j, r_j}(u_j, v_j) = \widehat h_{z_j, r_j}(u_j) - \widehat h_{z_j, r_j}(v_j) = (\widehat h_{z_j, d_{j, k}}(u_j) - 
	\widehat h_{z_j, d_{j, k}}(v_j) )  +  ({\widetilde h}_{z_j, d_{j, k}; r_j}(u_j) - {\widetilde h}_{z_j, d_{j, k}; r_j}(v_j) )\end{equation*}
where $\widetilde h_{z_j, d_{j, k}; r_j}(w) \coloneqq \int_{\partial B(z_j, r_j)} \widetilde h_{z_j, d_{j, k}}(y) 
\mathcal P_{z_j, r_j}(w, y) \sigma_{z_j, r_j}(dy)$ is the harmonic average of the field $\widetilde h_{z_{j}, 
	d_{j, k}}$ at $w$ w.r.t. the circle $\partial B(z_j, r_j)$ with $\mathcal P_{z, r}(\cdot, \cdot)$ being the 
Poisson kernel for the disk $B(z, r)$ and $\sigma_{z, r}(\cdot)$ being the uniform distribution on the circle $\partial B(z, r)$. Similar decomposition holds for $\Delta_{z_k, r_k}(u_k, v_k)$ as well. These averages are well-defined just as the circle averages by the continuity of Poisson kernel. Now notice that the random variable $\Delta_{z, r}(u, v)$ is a 
difference of two averages involving $h$ and hence is well-defined without any normalization, see, e.g., 
\cite[Section~3.2]{SheffieldWelding2016}). Hence, by the independence between 
$\widetilde h_{z, r}$ and $\mathcal F_{z, r}$, we get in view of this decomposition
\begin{equation*}
	\begin{split}
		\cov[\Delta_{z_j, r_j}(u_j, v_j),  \Delta_{z_k, r_k}(u_k, v_k)] &= \cov[\widehat h_{z_j, d_{j, k}}(u_j) - \widehat h_{z_j, d_{j, k}}(v_j),\,  \widehat h_{z_j, d_{j, k}}(u_k) - \widehat h_{z_j, d_{j, k}}(v_k)].
	\end{split}
\end{equation*}
Bounding this by the Cauchy-Schwarz inequality and subsequently the resulting variance terms by 
Lemma~\ref{lem:variance_estimate} yields us
\begin{equation}\label{eq:covariance_bnd}
	\big|\,\cov[\Delta_{z_j, r_j}(u_j, v_j),  \Delta_{z_k, r_k}(u_k, v_k)] \, \big| \le C 	\frac{r_j r_k}{d_{j, k}^2}.
\end{equation}
We now go back to \eqref{eq:var_expand} and consider, for some $j \in S_i$ and $\ell \ge i$, the sum of 
the (absolute) covariance terms
\begin{equation*}
	\sum_{k \in S_\ell}\big | \cov[\Delta_{z_j, r_j}(u_j, v_j),  \Delta_{z_k, r_k}(u_k, v_k)] \big| \le \sum_{t \in \N_{>0}} \, \sum_{k \in 
		T_{t, j, \ell}} \big | \cov[\Delta_{z_j, r_j}(u_j, v_j),  \Delta_{z_k, r_k}(u_k, v_k)] \big | 
\end{equation*}
where $T_{t, j, \ell} \coloneqq \{k \in  S_{\ell} : A(z_j, r_j + (t-1)L_0\lambda^{-k_\ell}, r_j + t 
L_0\lambda^{-k_\ell}) \cap B(z_k, r_k) \ne \emptyset\}$. Since the disks $B(z_k, r_k)$'s are disjoint; and $j 
\in S_{i}$ and $k \in S_{\ell}$, it follows that
\begin{equation}\label{eq:Ttjlbnd}
	|T_{t, j, \ell}| \le C \, \frac{(r_j + t L_0 \lambda^{-k_\ell} ) L_0 \lambda^{-k_\ell}}{L_0^2 \lambda^{-2k_\ell}} 
	\le \Cl{Ttjlbnd} (\lambda^{k_\ell - k_i}  +  t).
\end{equation}
As to $\cov[\Delta_{z_j, r_j}(u_j, v_j), \Delta_{z_k, r_k}(u_k, v_k)]$, using the same bounds on $r_j$ and $r_k$, we get from \eqref{eq:covariance_bnd}
\begin{equation*}
	\max_{k \in T_{t, j, \ell}} \big |\, \cov[\Delta_{z_j, r_j}(u_j, v_j), \Delta_{z_k, r_k}(u_k, v_k)] \, \big | \le C \frac{\lambda^{k_{\ell} - k_i}}{(\lambda^{k_\ell - k_i} + t)^2}\,.
\end{equation*}
Since this bound is decreasing in $t$ and $|S_\ell| \le 100 \nu$ by Definition~\ref{def:goodcollection}, we 
can write in view of \eqref{eq:Ttjlbnd},
\begin{equation}\label{eq:covsumsmallell}
	\begin{split}
		\sum_{t \in \N_{>0}}\,\sum_{k \in T_{t, j, \ell}} &\big | \, \cov[\Delta_{z_j, r_j}(u_j, v_j),  \Delta_{z_k, r_k}(u_k, v_k)] \, \big |\\  
		&\le C\lambda^{k_\ell - k_i} \sum_{t = 1}^{t_\nu}\frac{1}{\lambda^{k_\ell - k_i} + t} \le C \lambda^{k_\ell - k_i} \log 
		\frac{\lambda^{k_\ell - k_i} + t_\nu}{\lambda^{k_\ell - k_i}}
	\end{split}
\end{equation}
where $t_\nu$ is the smallest integer satisfying $\Cr{Ttjlbnd}\sum_{t = 1}^{t_\nu} (\lambda^{k_\ell - k_i} + t) 
\ge 100 \nu$ provided $\Cr{Ttjlbnd}(\lambda^{k_\ell - k_i} + 1) \le 100 \nu$. On the other hand, when 
$\Cr{Ttjlbnd}(\lambda^{k_\ell - k_i} + 1) > 100 \nu$, we can instead use the naive bound
\begin{equation}\label{eq:covsumbigell}
	\sum_{t \in \N_{>0}}\,\sum_{k \in T_{t, j, \ell}} \big | \, \cov[\Delta_{z_j, r_j}(u_j, v_j),  \Delta_{z_k, r_k}(u_k, v_k)] \, \big |  \le C \nu 
	\frac{\lambda^{k_\ell - k_i}}{(\lambda^{k_\ell - k_i} + 1)^2} \le C\nu \lambda^{-(k_\ell - k_i)}.
\end{equation}
Let us further analyze the bound in \eqref{eq:covsumsmallell}. It is clear from the definition of $t_\nu$ 
that $\Cr{Ttjlbnd}\sum_{t = 1}^{t_\nu} (\lambda^{k_\ell - k_i} + t) \le 200 \nu$ and hence $\lambda^{k_\ell - 
	k_i} t_\nu \le C \nu$. Consequently, we obtain from \eqref{eq:covsumsmallell}:
\begin{equation}\label{eq:covsumsmallell2}
	\begin{split}
		\sum_{t \in \N_{>0}}\,\sum_{k \in T_{t, j, \ell}} &\big | \, \cov[\Delta_{z_j, r_j}(u_j, v_j),  \Delta_{z_k, r_k}(u_k, v_k)] \, \big | \\ 
		&\le C \lambda^{k_\ell - k_i} \log (1  + C \nu \lambda^{-2(k_\ell -k_i)} ) \le C \nu \lambda^{-(k_\ell - k_i)}
	\end{split}
\end{equation}
where in the final step we bounded $\log (1 + x)$ by $x$. Notice that this bound is good as soon as $\nu 
\lambda^{-2(k_\ell - k_i)} \le 1$, i.e., $\nu \le \lambda^{2(k_\ell - k_i)}$ and we can bound the log term 
by $\log \e \nu$ otherwise. Using these bounds in different `regimes'' based on the value of $k_\ell$, we get 
(we suppress the constant prefactor $C$ by using ``$\preceq$'' instead of ``$<$'' in all steps)
\begin{equation*}
	\begin{split}
		\sum_{\ell \ge i}\,\sum_{k \in S_{\ell}} & \big | \, \cov[ \Delta_{z_j, r_j}(u_j, v_j),  \Delta_{z_k, r_k}(u_{k}, v_{k})] \, \big | \\ &\preceq 
		\sum_{\lambda^{2(k_\ell - k_i)} < \nu} \lambda^{k_\ell - k_i} \log \e \nu + \sum_{  \nu \le \lambda^{2(k_\ell - 
				k_i)}} \nu \lambda^{-(k_\ell - k_i)} \le C(\lambda_0) \sqrt{\nu} \log \e \nu
	\end{split}	
\end{equation*}
(recall that $\lambda \ge \lambda_0 > 1$). Plugging this as well as \eqref{eq:variance_bnd} into 
\eqref{eq:var_expand}, we can conclude \eqref{eq:exp_max_sum_bnd} in view of the bound $|S_i| \le [c\nu, C \nu]$ as implied by Definition~\ref{def:goodcollection}.
\end{proof}

Finally we give
\begin{proof}[Proof of Lemma~\ref{lem:variance_estimate}]
Due to the scale and translation invariance of $h$ (see~\eqref{eq:scale_trans_invar}), it suffices to prove 
the estimate for $z = 0$ and $r = 1$ which will drop from all the notations below. Since $\widehat h(\cdot)$ 
is harmonic in $B(0, 1)$, we can write in the notations we introduced in the previous proof (see the 
discussion after \eqref{eq:variance_bnd}):
\begin{equation*}
	\begin{split}
		\widehat{h}(u) - \widehat{h}(v) = \int_{\partial B(0, 1)} h(y)  ( \mathcal P(u, y) -  \mathcal P(v, y) ) \sigma(dy) .
	\end{split}	
\end{equation*}
As already noted, this average is well-defined even without any normalization and hence we have the following expression from the defining properties of the whole-plane GFF (see, 
e.g.~\cite[display~(3.4)]{SheffieldWelding2016}):
\begin{equation*}
	\begin{split}
		\E[ (\widehat{h}(u) - \widehat{h}(v))^2 ]  = \int_{\partial B(0, 1) \times \partial B(0, 1)}   ( \mathcal P(u, x) -  \mathcal P(v, x)) G(x, y)  ( \mathcal P(u, y) -  \mathcal P(v, y)) \sigma(dx) \sigma(dy)
	\end{split}
\end{equation*}
where $G(x, y) =  - \log |x - y|$. From this we can immediately deduce \eqref{eq:squared_diff} owing to the 
continuity of the Poisson kernel on $B(0, 1) \times \partial B(0, 1)$ (see, 
e.g.,~\cite[Theorem~3.44]{morters2010brownian}).
\end{proof}

\section{Length of LFPP geodesics} \label{sec:LFPP}
In this section we will prove Theorem~\ref{thm:LFPP_length}. Like in the case of Theorem~\ref{thm:main}, 
the principal tool is the following analogue of Proposition~\ref{prop:crossingprob}.
\begin{propbis}{prop:crossingprob}\label{prop:crossingprobbis}
For any $\lambda_0 > 1$ and $\xi > 0$, there exist $\rho = \rho(\lambda_0, \xi), \varepsilon_0 = \varepsilon_0(\xi) \in (0, 1)$ 
and a positive absolute constant $C$ such that the following holds. For any collection of disks $\{A_j\}_{j = 1}^n$ that is $(\lambda, \nu)$-balanced for some $\lambda \ge 
\lambda_0$ and $\nu > 1$ and has maximum and minimum radius in $[32\varepsilon^{1/2}, 1]$, one has
\begin{equation}\label{eq:crossingprobbis}
\P[ \text{a $\widehat{D}_h^{\varepsilon}$-geodesic crosses $A_{j, \circ}$ for all $j$} \,] \le C \rho^{\frac{m 
	\sqrt{\nu}}{\log \e \nu}}
\end{equation}
whenever $\varepsilon \in (0,  \varepsilon_0)$.
\end{propbis}
In order to deduce Theorem~\ref{thm:LFPP_length} from Proposition~\ref{prop:crossingprobbis}, we will 
need a ``finite-length'' analogue of Theorem~\ref{thm:hausdorff} which we now state. Let us 
recall the notion of straight runs and $(\lambda, k_0)$-sparsity from Section~\ref{subsec:LQG_geod_intro}. 
We call the straight runs of a path $P$ as $(\lambda, k_0)$-sparse {\em down to scale $\delta$} if $P$ does not exhibit any nested collection of straight runs on a sequence of scales $L_{k_1} > \cdots > L_{k_m}$ with 
$L_k \coloneqq L_0 \lambda^{-k}$, $L_{k_m} \ge \delta$ and  $m \ge \frac12 \max\{k_m, k_0\} = \frac12 
k_m$.
\begin{lemma}\label{lem:length_lower_bnd}
Let the distance between the endpoints of a given path $P$ be at least $L_0$. Also suppose that its straight runs are $(\lambda, k_0)$-sparse down to scale $\delta$ for some $\delta \in 
(0, L_0)$ and $k_0 > 0$. Then there exist $s = s(\lambda) > 0$ and $c = c(\lambda, L_0) > 0$ such that the 
Euclidean length of $P$ is at least $c\, \delta^{-s}$.
\end{lemma}

\begin{proof}
Choose $p \in [\lambda/2, \lambda]$ such that $\beta \coloneqq \sqrt{p(p+1)}$ is {\em strictly} larger than $\lambda$. Now using the same algorithmic construction as given in the proof of 
\cite[Lemma~5.2]{AizBur99}, we get a nested sequence $\Gamma_0, \ldots, \Gamma_{k_{{\rm max}}}$ of 
collections of segments of $P$ such that:
\begin{itemize}
	\item each $\Gamma_k$ is a collection of segments of diameter at least $L_k$;
	\item in each generation (as defined by $k$), distinct segments are at distances at least $\varepsilon L_k $ with $\varepsilon = (\lambda/  p)-1$;
	\item each segment of $\Gamma_k\,(k > 1)$ is contained in one of the segments of $\Gamma_{k-1}$, with 
	the number of immediate descendants thus contained in a given element of $\Gamma_{k-1}$ at least $p$ 
	and is at least $p + 1$ {\em unless} it exhibits a straight run at the scale $L_k$
\end{itemize}
(here $k_{{\rm max}}\coloneqq \max\{k:L_{k} \ge \delta\}$). Now, for points $x \in \cup_{\eta \in \Gamma_k} \eta$, define $n_k(x)$ to be the number of immediate descendants of the set containing $x$ within 
$\Gamma_{k-1}$ \cite[Lemma~5.4]{AizBur99}. Then, using the identity \cite[display~(5.19)]{AizBur99}, we can write
\begin{equation*}
1 = \sum_{\eta \, \in \, \Gamma_{k_{{\rm max}}}}\,  \prod_{j = 1}^{k_{{\rm max}}} n_j(\eta)^{-1} \le \, |\Gamma_{k_{{\rm max}}}| \cdot \Big( \min_{\eta \in \Gamma_{k_{{\rm max}}}} \prod_{j = 1}^{k_{{\rm max}}} n_j(\eta)\Big)^{-1}
\end{equation*}
where the number $n_j(n)$ is the constant value that $n_j(x)$ takes for $x \in \eta$. However, since the straight runs of $P$ are $(\lambda, k_0)$-sparse down to scale $\delta$, it follows from the aforementioned 
items that $\prod_{j=1}^{k_{{\rm max}}}n_{j}(x)\ge\beta^{k_{{\rm max}}}$ (see also the proof of Theorem~5.1 in \cite{AizBur99}). Therefore, the previous display 
gives us $|\Gamma_{k_{{\rm max}}}| \ge \beta^{k_{{\rm max}}}$. From this we can deduce,
\begin{equation*}
|P| \ge |\Gamma_{k_{{\rm max}}}| \cdot L_{k_{{\rm max}}}\ge \Big(\frac{\beta}{\lambda}\Big)^{k_{{\rm max}}}L_{0} = c(\lambda, p, L_0)\delta^{-s}
\end{equation*}
with $s = c \log(\beta/\lambda)/\log(\lambda) > 0$ for some suitable absolute constant $c > 0$ upon noting 
that $L_{k_{{\rm max}}} <  \delta\lambda$.
\end{proof}
	
\smallskip

We can now finish the 
\begin{proof}[Proof of Theorem~\ref{thm:LFPP_length}]
The proof follows in essentially the same fashion from Proposition~\ref{prop:crossingprobbis} and 
Lemma~\ref{lem:length_lower_bnd} as did the proof of Theorem~\ref{thm:main} from  
Proposition~\ref{prop:crossingprob} and Theorem~\ref{thm:hausdorff} by controlling the probability of the 
sparsity of straight runs in LFPP geodesics down to scale $32\varepsilon^{1/2}$.
\end{proof}

\medskip

Finally, it remains to give the:
\begin{proof}[Proof of Proposition~\ref{prop:crossingprobbis}]
Let us consider the following redefinition of the events in \eqref{def:tildeF_H} for some $M > 0$ which will be 
later set to an absolute constant:
\begin{equation}\label{def:tildeF_H2}
\begin{split}
\widetilde F_{z, r} &\coloneqq 	\big\{ \widehat D_{\widetilde h_{z, r}}^{\varepsilon}(\text{around $A(z, 3r/4, 
7r/8)$})  \ge \e^{-\xi M} \widehat D_{\widetilde h_{z, r}}^{\varepsilon}(\text{across $A(z, r/2, 5r/8)$})\big\}, \text{ and}\\H_{z, r} & \coloneqq \big\{\sup_{u, v \in B(z, 15r/16)} \, |\widehat h_{z, r} (u) -  \widehat h_{z, r}(v) 
|  \ge M \big\}
\end{split}
\end{equation}
where $\widehat D_{\widetilde h_{z, r}}^{\varepsilon}$ is the LFPP metric constructed from $\widetilde h_{z, r}$ in the same way as $\widehat D_{h}^{\varepsilon}$ was constructed from $h$ 
(recall~\eqref{def:widehat_h_ep}). We will follow this convention of denoting LFPP metrics in the sequel. 
Similarly we redefine the event $F_{z, r}$   from the statement of Lemma~\ref{lem:crossingonedisk} (see 
also \ref{property1}) to involve $\widehat D_h^{\varepsilon}$-geodesics instead of $D_h$-geodesics.  In view of \eqref{eq:bnd_F}, 
Lemma~\ref{lem:tildeF_H} and Lemma~\ref{lem:max_dev} (with $H_{z, r}$ redefined as above for which the 
proof given in the previous section works similarly for a possibly new choice of the absolute constant $M$), it suffices to verify the properties 
\ref{property1}--\ref{property3} for  $\{\widetilde F_{z_j, r_j}\}_{j = 1}^n$ and $\{H_{z_j, r_j}\}_{j = 
	1}^n$ which we do in the remainder of the proof. A 
crucial observation to make here --- which will be used repeatedly in the sequel often without any explicit reference --- is that the random variable $\widehat h_{\varepsilon}^{\star}(z)$  is a {\em radially symmetric 
	integral} of (and therefore is measurable relative to) $h\vert_{B(z, \varepsilon^{1/2})}$ for all $z \in \mathbb C$. This is the main reason why we chose to work with 
$\widehat h_{\varepsilon}^{\star}$ instead of $h_{\varepsilon}^{\star}$.

\vspace{0.3cm}

\noindent{\em Property~\ref{property1}.}	This follows from the same reasoning as given in the proof of 
Lemma~\ref{lem:tildeF_H} for the same property except that in place of Weyl scaling, we can directly use the 
definition of the metric as given by \eqref{eq:D_h_ep} (with $\widehat h_{\varepsilon}^\star$ in place of 
$h_{\varepsilon}^\star$). For this last part of the argument, it is crucial that $F_{z, r}$ be measurable relative 
to $h\vert_{B(z, 15r/16)}$ which is true if $r \ge 32\varepsilon^{1/2}$.

\vspace{0.25cm}

\noindent{\em Property~\ref{property2}.} Again the same argument as in the proof of Lemma~\ref{lem:tildeF_H} works in view of the lower bound on the minimum radius of the disks.

\vspace{0.25cm}

\noindent{\em Property~\ref{property3}.} This part requires a delicate argument since, unlike in the case of 
Lemma~\ref{lem:tildeF_H}, we do {\bf not} have $\P[\widetilde F_{z, r}] = \P[\widetilde F_{0, 1}]$. The way to 
get around this is to use the {\em tightness} of the (re-scaled) LFPP metrics \cite{ding2020tightness}. 
However, to the best of our knowledge, there is no ``readily available'' tighness results for $\mathfrak 
a_{\varepsilon}^{-1} \widehat D_{\widetilde h_{z, r}}^{\varepsilon}(\text{across/around
$A(z, \cdot, \cdot)$})$ ($\varepsilon \in (0, 1)$) in particular (see 
Section~\ref{subsec:LFPP_intro} for $\mathfrak a_{\varepsilon}$) and some work is necessary to {\bf transfer} the tightness bounds  from $\mathfrak a_{\varepsilon}^{-1}  D_{h}^{\varepsilon}(\text{across/around A}) $ to these 
random variables. Consequently, in the 
course of our proof, we will switch between and compare events similar to $\widetilde F_{z, r}$ defined for 
several LFPP metrics and hence it would be convenient to have a generic notation for such events. To this 
end let us define, for any random, {\em intrinsic} metric $D$ defined on $B(z, r)$ and prefactor ${\rm C} > 0$,
\begin{equation*}
\widetilde F_{z, r}(D, {\rm C}) \coloneqq 	\big\{ D(\text{around $A(z, 3r/4, 7r/8)$})  
\ge {\rm C}\, D(\text{across $A(z, r/2, 5r/8)$})\big\}
\end{equation*}
(cf.~the definition of $\widetilde F_{z, r}$ in \eqref{def:tildeF_H2}). 

\vspace{0.17cm} 

Now suppose for the moment being that we also have,
\begin{equation}\label{eq:tildeFlocal}
\P\Big [ \widetilde F_{z, r}\Big(\widehat D_{\widetilde h_{z, r}}^{\varepsilon},\, \Cl{prefact}(\xi) \Big) \Big] \le \frac 12
\end{equation}
for all $r \in (32 \varepsilon^{1/2}, 1)$, $\varepsilon \in (0, c(\xi))$ and some $\Cr{prefact}(\xi) > 1$. 
In order to verify \ref{property3}, we 
need to obtain a non-trivial  upper bound (depending only on $\xi$) for the probability of $\widetilde F_{z, r}\big(\widehat D_{\widetilde h_{z, 
r}}^{\varepsilon},\, \e^{-\xi M} \big)$. To this end, let us consider a non-negative, radially 
symmetric (bump) function $\phi \in C_c^{\infty}(\mathbb C)$ supported in $B(0, 23/32)$ which is identically 
equal to 1 on $B(0, 21/32)$.  Then using similar argument as that for the upper bound on $\P[G_{0, 1}]$ in 
the proof of Lemma~\ref{lem:crossingonedisk} and the fact that $\widehat {\mb h}_{\varepsilon}^\star(z)$ is determined by $\mb h\vert_{B(z, \varepsilon^{1/2})}$ for any distribution $\mb h$ (provided 
all the distributional pairings are well-defined), we get a constant $\Cl{prefact2} = 
\Cr{prefact2}(\Cr{prefact}, M) > 0$ such that
\begin{equation}\label{eq:hatDh_event}
\widetilde F_{z, r}\Big(\widehat D_{\widetilde h_{z, r}}^{\varepsilon},\, \e^{-\xi M} \Big)  = \widetilde F_{z, 
r}\Big(\widehat D_{\widetilde h_{z, r} + \Cr{prefact2}\, \phi(r^{-1} (\cdot\,  - \, z))}^{\varepsilon},\, \Cr{prefact} \Big)
\end{equation}
whenever $r \ge 32\varepsilon^{1/2}$. Denoting by $\mb P$ and $\widetilde{\mb P}$ the laws of $\widetilde 
h_{z, r}$ and $\widetilde h_{z, r} + \Cr{prefact2}\,\phi(r^{-1} (\cdot\,  - \, z)$ respectively, let us recall the following fact which is a 
consequence of Jensen's inequality, see e.g.~
the discussion following (2.7) in \cite{BDZ95} for a proof. For any event $A$ with positive 
$\widetilde{\mb P}$-probability, one has
\begin{equation}\label{eq:gen.entrop.lower}
	\mb P[A]\geq \widetilde{\mb P}[A]e^{-(1/\widetilde{\mb P}[A])(H(\widetilde{\mb P}|\mb P)+1/e)}.
\end{equation}
where $H(\widetilde{\mb P}|\mb P):=\widetilde{\mb E}\left[\log\frac{d\widetilde{\mb P}}{d\mb P}\right]$ is the 
relative entropy of $\widetilde{\mb P}$ with respect to $\mb P$. The Radon-Nikodym derivative 
$\frac{d\widetilde{\mb P}}{{d\mb P}}$, in this case, is given by the formula (see, 
e.g.~\cite[Proposition~3.4]{MillerSheffield_Imaginary12016})
\begin{equation*}
\frac{d\widetilde{\mb P}}{{d\mb P}} = \exp \left( (h, \phi_{z, r})_{\nabla} - \frac 12 (\phi_{z, r}, \phi_{z, r})_{\nabla} \right)
\end{equation*}
where $(f, g)_{\nabla} \coloneqq \int_{\mathbb C} \nabla f \cdot \nabla g \, d^2 z$ is the Dirchlet inner 
product and $\phi_{z, r}(\cdot) \coloneqq \Cr{prefact2}\, \phi(r^{-1} (\cdot\,  - \, z))$. From this we can now compute the relative entropy in terms of $(\phi_{z, r}, \phi_{z, r})_{\nabla}$ as follows:
\begin{equation*}
H(\widetilde{\mb P}|\mb P) = \widetilde {\mb E}	[ (\mb h, \phi_{z, r})_{\nabla} - \frac 12 (\phi_{z, r}, \phi_{z, 
r})_{\nabla} ]  = {\mb E}	[ (\mb h + \phi_{z, r}, \phi_{z, r})_{\nabla} - \frac 12 (\phi_{z, r}, \phi_{z, r})_{\nabla} ] = \frac 12 (\phi_{z, r}, \phi_{z, r})_{\nabla}
\end{equation*}
(this is a special case of the famous Cameron-Martin theorem, see, e.g.~\cite[Theorem~14.1]{janson_1997}). 
By standard change of variable one gets that $(\phi_{z, r}, \phi_{z, r}) _{\nabla} = \Cr{prefact2}^2\,(\phi, \phi)_{\nabla}$ and 
hence applying \eqref{eq:gen.entrop.lower} to the event $\widetilde F_{z, r}(\widehat D_{\mb h}^{\varepsilon}, \Cr{prefact})^c$, we obtain from \eqref{eq:tildeFlocal} and \eqref{eq:hatDh_event}
\begin{equation*}
\P\big[\widetilde F_{z, r}\big(\widehat D_{\widetilde h_{z, r}}^{\varepsilon},\, \e^{-\xi M} \big) \big] \le c(\xi) < 1
\end{equation*}
which yields \ref{property3}.

\vspace{0.17cm} 

Let us now show \eqref{eq:tildeFlocal} for which we again start with an intermediate statement. So suppose that 
\begin{equation}\label{eq:tildeFfull}
\P\big [ \widetilde F_{z, r}\big(\widehat D_{h}^{\varepsilon},\, \Cl{prefactfull}(\xi) \big) \big] \le \frac 14
\end{equation}
for some $\Cr{prefactfull}(\xi) > 0$, all $r \in (\varepsilon, 1)$ 
and $\varepsilon \in (0, c(\xi))$. From this, \eqref{eq:tildeFlocal} immediately follows by controlling the fluctuation of the harmonic field $\widehat h_{z, r}$ on $B(z, 15r/16)$ using 
Lemma~\ref{lem:variance_estimate}, \eqref{eq:max_exp_single_ball} and the Borell-Tsierlson inequality 
(see~\eqref{eq:borel_tsierl}). So we focus on \eqref{eq:tildeFfull} in the remaining part of the proof. 

\vspace{0.17cm} 

Continuing in the same vein, we first note that
\begin{equation}\label{eq:tildeFDh}
\P\big [ \widetilde F_{z, r}\big(D_{h}^{\varepsilon},\, \Cl{prefactDh}(\xi) \big) \big] \le \frac 18
\end{equation}
for some $\Cr{prefactDh}(\xi) > 0$ and all $r > \varepsilon$ as the the random variables $$\mathfrak 
a_{\delta}^{-1}D_h^{\delta}(\text{around $A(3/4, 7/8)$}) \text{ and } \big(\mathfrak 
a_{\delta}^{-1}D_h^{\delta}(\text{across $A(1/2, 5/8)$})\big)^{-1}$$ are tight for $\delta \in (0, 1)$ \cite[Proposition~4.1]{ding2020tightness}.  Let us explain the connection between these two in more detail. 
It follows from the definition of $D_h^{\varepsilon}$ that 
\begin{equation*}
D_h^{\varepsilon}(rz, rw) =  r D_{h(r\cdot)}^{\varepsilon/r}(z, w), \quad \forall z, w \in \mathbb C
\end{equation*}
(\cite[Lemma~2.6]{DubFalcGwynPfeSun20}). Combined with the scale and translation invariance of $h$ \eqref{eq:scale_trans_invar}, this yields \eqref{eq:tildeFDh} in view of the above-mentioned tightness of $D_h^{\delta}$ over $\delta \in (0, 1)$ (see also \cite[Definition~1.6 -- Axiom~V]{pfeffer2021weak}).

\vspace{0.17cm} 

Let us now finish the proof of \eqref{eq:tildeFfull} which follows immediately from \eqref{eq:tildeFDh} and the 
following (local) uniform comparison result proved in \cite[Lemma~2.1]{DubFalcGwynPfeSun20}:

\vspace{0.1cm}

Almost surely, $\lim_{\delta \to 0} \tfrac{\widehat 
	D_h^{\delta}(z, w; B(0, 1))}{D_h^{\delta}(z, w; B(0, 1))} = 1$ uniformly over all $z, w \in B(0, 1)$ 
($z \ne w$). 
\end{proof}

\bibliographystyle{amsalpha}
\newcommand{\etalchar}[1]{$^{#1}$}
\providecommand{\bysame}{\leavevmode\hbox to3em{\hrulefill}\thinspace}
\providecommand{\MR}{\relax\ifhmode\unskip\space\fi MR }
\providecommand{\MRhref}[2]{%
	\href{http://www.ams.org/mathscinet-getitem?mr=#1}{#2}
}
\providecommand{\href}[2]{#2}

\end{document}